\documentclass[12pt,a4paper,reqno,oneside]{amsart}

\usepackage{csquotes}
\usepackage{vmargin}
\setpapersize{A4}
\setmargins
{2.5cm}		
{1.5cm}		
{16cm}		
{23.42cm}	
{10pt}		
{1cm}		
{0pt}		
{2cm}		

\usepackage[dvipsnames]{xcolor}
\usepackage[%
bookmarks=true,			
unicode=true,			
colorlinks=true,		
linkcolor=Blue,			
citecolor=Blue,			
filecolor=magenta,		
urlcolor=RoyalBlue		
]{hyperref}				
\usepackage[all]{xy}	%
 
\CompileMatrices		

\UseTips				

\input xypic
\usepackage[bookmarks=true]{hyperref}	

\usepackage{amssymb,latexsym,amsmath,amscd}
\usepackage{xspace}
\usepackage{color}
\usepackage{graphicx}
\usepackage{dsfont}
\reversemarginpar

\theoremstyle{plain}
\newtheorem{theorem}{Theorem}[section]
\newtheorem*{theorem*}{Theorem}
\newtheorem{proposition}[theorem]{Proposition}
\newtheorem{corollary}[theorem]{Corollary}
\newtheorem{lemma}[theorem]{Lemma}

\theoremstyle{definition}
\newtheorem{definition}[theorem]{Definition}

\newtheorem{remark}[theorem]{Remark}

\newcommand{\enm}[1]{\ensuremath{#1}}     %
\newcommand{\cal}[1]{\mathcal{#1}}

\newcommand{\CC}{\enm{\mathbb{C}}}
\newcommand{\NN}{\enm{\mathbb{N}}}
\newcommand{\RR}{\enm{\mathbb{R}}}

\newcommand{\ZZ}{\enm{\mathbb{Z}}}
\newcommand{\PP}{\enm{\mathbb{P}}}
\newcommand{\Mm}{\enm{\cal{M}}}

\newcommand{\Hh}{\enm{\cal{H}}}
\newcommand{\Ii}{\enm{\cal{I}}}
\newcommand{\Oo}{\enm{\cal{O}}}

\renewcommand{\phi}{\varphi}

\newcommand{\ce}{\mathrel{\mathop:}=}

\author{E. Ballico, E. Gasparim, {\tiny and} B. Suzuki}
\title{Hyperelliptically fibred surfaces with nodes} 

\begin{document}
\maketitle

\begin{abstract} 
Using elementary methods of algebraic geometry, we present constructions of hyperelliptically fibred surfaces containing nodal fibres. 
\end{abstract}

\tableofcontents

\maketitle

\section{\bf Motivation}
Hyperelliptically
fibred surfaces appear often in string theory and when they contain singular fibres these provoke the existence 
of $D$-branes. Elliptic fibrations are quite popular and well understood, while the precise role hyperelliptic fibrations
might play in string theory remains to be understood. 
Recall that a curve $C$ of genus $\geq 2$ is said to be hyperelliptic if there exists a morphism $C \rightarrow \mathbb P^1$ of degree 2.
Examples of string theoretical work considering hyperelliptic fibrations are \cite{mmp,xy}. 
We were asked by physicists to provide a summary of results on the existence of hyperelliptic fibrations with certain 
types of singularities, 
and this was the motivation to produce this note. 
Most results here were collected from the standard algebraic geometric literature, 
however we do provide original constructions in sections \ref{example-eisen}, \ref{example-node}, and \ref{toric-families}
showing explicit examples of genus 2 families acquiring nodes. Nodes may occur in any of the following ways, as illustrated in Figure \ref{examples}:
 irreducible of genus 2 with a single node, irreducible of genus 2 with more than 1 nodes, reducible.
 
\begin{figure}[h]
\includegraphics[scale=0.12]{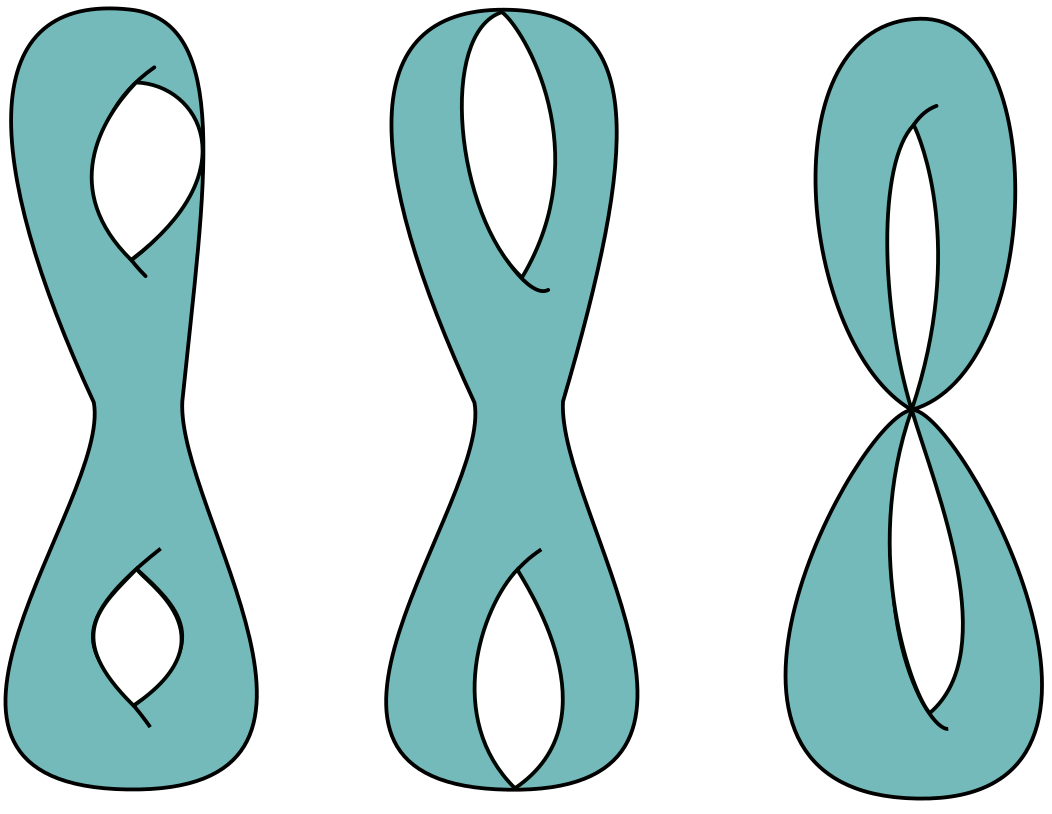}
\caption{Types of nodes on genus 2 curves}
\label{examples}
\end{figure}

In the review parts of sections \ref{sing-fibres} and \ref{smooth-fibres} our goal was to collect properties of numerical invariants of fibred surfaces, with focus on inequalities 
involving bounds on Chern numbers. Physics properties dictate that positivity of Chern numbers will imply the existence of D-branes, 
but this theme will remain for another paper dedicated to applications to F-theory.

Reversing a construction of Eisenbud, in subsection \ref{example-eisen} 
we find a nice family of all genus $g$ hyperelliptic curves together with also nodal ones
in the same parameter space, for example such a family includes the irreducible curve with a prescribed $\le g$ number of nodes and the union of an elliptic curve and a $g-1$ hyperelliptic curve.

Given that the motivation coming from string theory emphasises applications of genus $2$ curves, 
in subsection \ref{example-node} we present a second construction specifically tailored for the case of genus $2$. 
 Our second construction produces a large family 
 such that a general element is a smooth curve of genus $2$,
 it contains a $2$ parameter family
 of integral nodal curves with arithmetic genus $2$ and exactly one node, 
while it also contains a 1 parameter family of curves which are a nodal union of $2$ elliptic curves with a unique common point.
 In subsection \ref{toric-families}, we present some further details of the toric case. We observe also that an entire chapter about hyperelliptic curves will appear in the upcoming book by Eisenbud and Harris \cite{EH}.
 
 In section \ref{smooth-fibres} we summarise some properties of  surfaces fibred by smooth fibres. The most interesting case 
 being that of the very interesting Kodaira surfaces which are {\it not} locally trivial. Therefore a Kodaira 
 surface is not a fibre bundle, hence not all fibres being isomorphic. Furthermore, each fibre
 of a Kodaira surface is isomorphic to only finitely many fibres. Such wild local behaviour made it almost impossible 
 for us to illustrate section \ref{smooth-fibres} with figures in the same way as we did in section \ref{sing-fibres}.
A reader with better suggestions of how to draw a Kodaira surface is welcome to share their idea with us.

\section{\bf Fibrations containing singular fibres}\label{sing-fibres}

Let $X$ be a smooth compact complex surface.
The {\bf holomorphic Euler characteristic} of $X$ is $$\chi (\Oo_X)\ce 1-h^1(\Oo_X) +p_g(X) = h^0(\Oo_X) -h^1(\Oo_X) +h^2(\Oo_X),$$ 
while the {\bf topological Euler characteristic} of $X$ is
$$e(X)= h^0(TX) -h^1(TX) +h^2(TX)-h^3(TX) +h^4(TX) $$ for any smooth compact manifold of real dimension 4.
 
For a smooth compact complex surface $X$ (i.e. $\dim_{\mathbb C}X=2$), Noether's formula gives 
 $$\boxed{12\chi(\Oo_X) =c_1^2(X)+c_2(X)= (K\cdot K) +e(X)}$$ 
 where $K$ is the canonical divisor class \cite[Eq.(4)p.26]{bhpv}\cite[p.472]{gh}.
 In contrast, we also observe that if $C$ is a smooth compact curve (i.e. a Riemann surface) and
 hence the two numbers $h^0(\Omega ^1_C)$ and $h^1(\Oo_C)$ are the same, 
 we have that the topological and holomorphic Euler characteristics satisfy
 $\boxed{2\chi(\Oo_C)=e(C) }.$ 
 
 \begin{remark}\label{a1}
Let $X'$ be the surface obtained from $X$ blowing up one point. We have 
that the irregularities satisfy $h^1(\Oo_{X'})=h^1(\Oo_X)$,
the plurigenera satisfy $p_g(X') =p_g(X)$ and hence the holomorphic Euler characteristics satisfy $\chi (\Oo_{X'}) =\chi(\Oo _X)$.
We have $c_1^2(X') =c_1^2(X)-1$ and consequently $c_2(X') =c_2(X)+1$, see \cite{bhpv}.
\end{remark}

We recall the fundamental local-triviality theorem of Grauert--Fischer.
\begin{theorem}\label{a2}
\cite[p.\thinspace 36]{bhpv} Let $f\colon X\to Y$ be a smooth holomorphic family of compact complex manifolds. The holomorphic map $f$ is locally trivial (in the Euclidean topology) over $Y$ if and only if all fibres of $f$ are biholomorphic.
\end{theorem}
Note that in Theorem \ref{a2} we do not assume that $Y$ is compact, we only assume that $f$ is a smooth and proper holomorphic map.
Let $f\colon X\to D$ be a proper holomorphic map with $X$ a smooth and connected complex surface (often called a fibration of curves with $D$ as a base) and $D$ a (not necessarily compact) Riemann surface (even not compact algebraic, e.g. a disc of $\CC$). 
The sheaf $f_\ast (\Oo_X)$ is a locally free $\Oo_D$-sheaf, say of rank $r\ge 1$ with $\Oo_D$ as a direct factor and
$f = h_1\circ f_1$, where $D_1$ is a Riemann surface, $h_1\colon D'\to D$ is a finite holomorphic map of degree $r$, 
$f_1\colon X\to D_1$ is proper and $f_{1\ast}(\Oo _X)=\Oo_{D'}$. The map $h_1$ is the identity if $r=1$. All fibres of $f_1$ are connected, while
a general fibre of $f$ has exactly $r$ connected components. Most books say that $f$ is a fibration only if $f_\ast (\Oo_X)=\Oo_D$, because the general case is reduced to a fibration by taking a finite map of Riemann surfaces \cite[Ch.\thinspace III, \S 8]{bhpv}.
Assume $f_\ast (\Oo_X)=\Oo_D$, i.e. assume that a general fibre is connected (and so all fibres of $f$ are connected). Let $A$ be any smooth fibre of $f$ and $A_s\ce f^{-1}(s)$ any fibre of $f$. 
Then the Euler numbers satisfy 
$e(A_s)\ge e(A)$. If $X$ is compact, i.e. if $D$ is compact, then  by \cite[Prop.\thinspace III.11.4]{bhpv} we have

\begin{equation}\label{equal}
e(X) =e(A)e(D) +\sum _{s\in D} (e(A) -e(A_s)).
\end{equation}
Moreover, there is an easy criterion which gives $e(A_s)>e(A)$ if $A_s$ is singular and not a multiple of a smooth elliptic curve \cite[III.11.5]{bhpv}. Hence if $D$ is compact of genus $g_2$ and a general fibre of $f$ is smooth of genus $g_1$, then 
$$\boxed{e(X)\ge 4(g_1-1)(g_2-1)}$$ 
(recall here that the Euler number equals the top Chern number for compact $X$)
with strict inequality if some of the fibres are singular and either $g_1\ne 1$ or $g_1=1$, but there is at least one fibre with is not a multiple of an elliptic curve \cite[III.\thinspace 11.6]{bhpv}.
See \cite[\S III. 9]{bhpv} for conditions that prevent the existence of multiple fibres, while their for existence see \cite{mi}. 

\begin{remark}
Let $X$ be a smooth and connected complex surface. We are  interested in the case in which there is a surjective 
holomorphic map $f\colon X\to D$ with $D$ a smooth projective curve
and as general fibre of $f$ a smooth curve of genus $g\ge 2$. A surface $X$ which has such a fibration is always a projective surface for the following reasons, all due to Kodaira.
Let $\CC(X)$ be the field of all meromorphic functions on $X$. The field $\CC(X)$ is a finitely generated extension of the field $\CC$ and it has transcendental degree at most $2$, see \cite[Thm.\thinspace I.7.1]{bhpv}.
The transcendence degree $a(X)$ of the field $\CC(X)$ over $\CC$ is called the algebraic dimension of $X$. If $a(X)=0$, then $X$ has only finitely many curves and, in particular, there is no surjective map $f\colon X\to D$ with $D$ a curve \cite[Thm.\thinspace IV.8.2]{bhpv}, \cite[Thm.\thinspace 5.1]{kod1}. 
Any surface with $a(X)=1$ admits a fibration $u\colon X\to D$ with $D$ a smooth curve and such that 
a general fibre of $u$ is an elliptic curve, see \cite[VI.5.1]{bhpv},\cite[Thm.\thinspace 4.1]{kod1}, and any irreducible curve $T\subset X$
 is contained in a fibre of $u$ \cite[Thm.\thinspace 4.3]{kod1}.
Surfaces with $a(X) =2$ are algebraic and projective,  see\cite[Cor.\thinspace IV.6.5]{bhpv}, \cite[Thm.\thinspace 3.1]{kod1}.
\end{remark}


\subsection{Elliptic fibrations and elliptic surfaces}
Let $D$ be a compact complex curve of genus $g_2\ge 0$ and $f\colon X\to D$ be a proper holomorphic map with $X$ a compact complex surface 
and such that a general fibre of $f$ is an elliptic curve.

 Let $S$ be the set of all $s\in D$ such that $f^{-1}(s)$ is singular (it may even be a multiple fibre). For any smooth genus $1$ curve, let $j(E)\in \mathbb C\setminus \{0\}$ be its $j$-invariant. Recall that the $j$-invariant of an elliptic curve 
 $$y^2= x^3+vex+g$$ is 
 $$j(\tau) = \frac{4(24f)^3}{\Delta}, \quad \Delta = 4f^3 +27g^2,$$
 and two smooth elliptic curves are isomorphic if and only if they have the same $j$-invariant. 
 
For each $s\in D$ set $E_s\ce f^{-1}(s)$ and note that it is connected
\cite[p.\thinspace 200--216]{bhpv}. We will analyse two cases separately, depending on whether there 
are singular fibres or not.
\subsubsection{ Case $S=\emptyset$ }

In the case where there are no singular fibres (as illustrated in Figure \ref{nosing}), i.e. assuming that $f$ is a submersion, by (\ref{equal}) we have $e(X)=0$. No way to change that. 

\begin{figure}[h]
\includegraphics[scale=0.18]{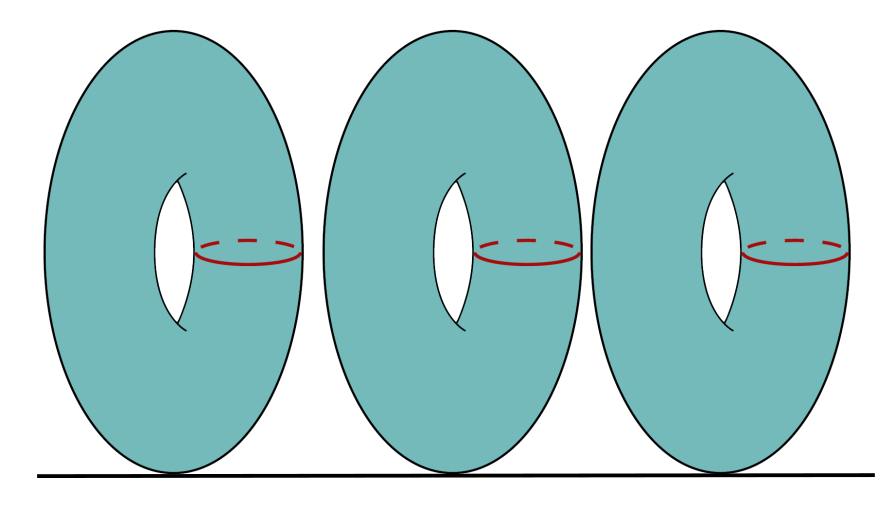}
\caption{Family without singular fibres}
\label{nosing}
\end{figure}

Next, we will use the following result.
\begin{theorem}\label{a3}
\cite[III.\thinspace 15.4]{bhpv} If $X$ is compact, i.e. if $D$ is compact, $f$ has no singular fibres and either $D\cong \mathbb P^1$ or $D$ is an elliptic curve, then $f$ is locally trivial over the base and in particular all fibres of $f$ are biholomorphic.
\end{theorem}

\begin{remark}\label{a4}
 Assuming $g_2\le 1$ for the base of an elliptic fibration, then by Theorem \ref{a3} we obtain that all fibres of $f$ are isomorphic to the same elliptic curve, $E$, and there is a finite open covering $\{U_i\}$ of $D$ such that $f^{-1}(U_i)\cong U_i\times E$ for all $i$. 
 \end{remark}
 
Kodaira gave a classification of the possible fibres occurring over $S$ \cite[Ch.\thinspace 7]{Fr}.
\subsubsection{Case $S\ne \emptyset$} 

For a description of all possible singular fibres (these are of 8 types) see \cite[Thm.\thinspace 6.2]{ko}. One case is illustrated in Figure \ref{sing01}.

\begin{figure}[h]
\includegraphics[scale=0.18]{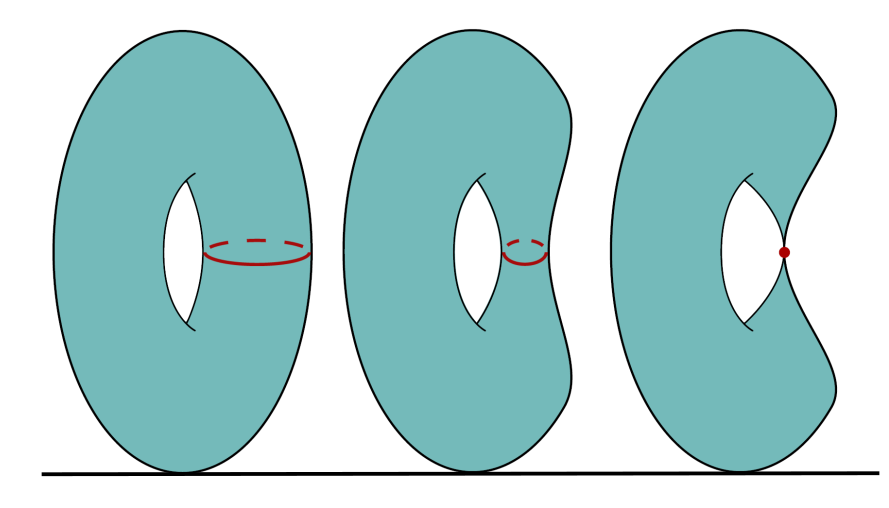}
\caption{Elliptic family with 1 node}
\label{sing01}
\end{figure}

\begin{definition}\label{min}
A fibration $f\colon X\to D$ is said to be {\bf minimal} (or relatively minimal) if no fibre of $f$ contains a $(-1)$-curve. 
\end{definition}

 We assume that the elliptic fibration $f\colon X\to D$ is   relatively minimal. The general case is obtained from the relatively minimal one, by making a finite sequence of blowing ups of points. For a relatively minimal elliptic surface $X$ we have $K_X^2=0$ 
 and hence the second Chern class becomes expressed in terms of the Euler characteristics and the degree $d$ of the dual of the line bundle of $R^1f_\ast(\mathcal O_X)$ \cite[Corollary 16]{Fr}:
$$\boxed{c_2(X) =12\chi(\mathcal O_X) = 12d \quad \textnormal{for an elliptic fibration}.}$$
Furthermore, $d\ge 0$ and $d=0$ if and only if the only singular fibres are multiple fibres whose reduction is smooth.

Therefore, we have noticed here that the presence of singular fibres in the case of elliptic fibrations immediately 
imply that the second Chern number is positive. 


\subsection{Fibrations by curves of genus $2$}\label{Sf2}

In this section we take a surface $X$ fibred by curves of genus  $g_1=2$ over a curve of genus $g_2$, and use the results of Xiao \cite{xiao1}. 
Recall that $X$ is projective for free. 
We set $q(X)\ce h^1(\Oo_X)=h^0(\Omega _X^1)$, and $p_g(X) \ce h^0(\omega _X) =h^2(\Oo_X)$. 
Since $\chi(\Oo _X) \ce1-q(X)+p_g(X)$,
we have 
$$\boxed{\chi(\Oo_X)\ge g_2-1 \quad \text{for}\quad g_1=2}$$
and {\it equality holds if and only $f$ has no singular fibre and all fibres are isomorphic}, see \cite{bea} or \cite[p.\thinspace 7]{xiao1}.

Recall that $g_2 \le q\le g_2+2$ and that $X$ is known if $q=g_2+2$.
 Recall that $E:= f_{\ast}(\omega _{X/D})$ is a rank $2$ vector bundle on $D$ with nonnegative degree and with all rank $1$ 
 quotients of nonnegative degree. Let $E_1$ a rank $1$ subsheaf of $E$ with maximal degree. The maximality of the integer $\deg(E_1)$ means that the coherent sheaf $E/E_1$ has no torsion. Since $D$ is a smooth curve and $E/E_1$ is a rank $1$ torsion free sheaf, $E/E_1$ is a line bundle. Set $\epsilon:= \deg(E_1)-\deg(E/E_1)$. The vector bundle $E$ is said to be unstable (resp. properly semistable, resp. stable) if and only if $\epsilon >0$ (resp. $\epsilon =0$ (resp. $\epsilon <0$) (in \cite[Ex. V.2.8]{h}) the integer $e$ is our integer $-\epsilon$). Since $\deg (E) =\deg(E_1) +\deg (E/E_1)$
 and $\epsilon = \deg(E_1)-\deg(E_2)$, we have $\epsilon \equiv \deg (E)\pmod{2}$. If $\epsilon >0$, i.e. if $E$ is stable, then the maximal degree rank $1$ subsheaf $E_1$ of $E$ is unique, because it corresponds to a unique section of the associated ruled surface over $D$ with negative self-intersection \cite[Prop.\thinspace V.2.21]{h}.
 The $\PP^1$-bundle $P$ is associated to a rank $2$ vector bundle $F$ on $D$ 
which is strongly related to $E$. If $\epsilon >0$ this is the Hirzebruch surface $F_\epsilon$ and the uniqueness of $E_1$ corresponds to the fact that the ruled surface $F_\epsilon \to \PP^1$, $\epsilon>0$, has a unique section $D_0$ with negative self-intersection, i.e. with $D_0^2=-\epsilon$ \cite[Th. V.2.17]{h}.

The following result is a summary of \cite[p.\thinspace 16]{xiao1}, see eq. (9), Thm.\thinspace 2.1 and the Rmq. that follows it.

\begin{theorem}\label{f21}
We have $\epsilon \le p_g+1$, $\epsilon 
\equiv \chi+g_2-1\pmod{2}$, $-g_2 \le \epsilon \le \chi-g_2+1$. If $q>g_2$, then $\epsilon =\chi-g_2+1$; $q=g_2+1$ if and only if $\epsilon =p_g+1-2g_2$. 
\end{theorem}
With this notation, $\epsilon >0$ is equivalent to the instability of the vector bundle $F$ and in this case there is a unique section $D_0$ of the $\PP^1$-bundle with negative self-intersection. There is also an effective divisor $R\subset P$ \cite[p.\thinspace 12]{xiao1}.

\begin{theorem}\label{f22}\cite[Thm.\thinspace 2.2]{xiao1}

\quad (i) Assume $\epsilon >0$ and $R\supseteq D_0$. Then
$$2\chi+6(g_2-1) \le K_X^2 \le 3\chi+5(g_2-1)-2\epsilon $$ and hence
$$\epsilon \le (\chi-g_2+1)/2.$$

\quad (ii) If either $\epsilon \le 0$ or else both $\epsilon >0$ and $R\not\supseteq D_0$, then
$$\max\{2\chi+6(g_2-1),\chi+7(g_2-1)+3\epsilon \} \le K_X^2 \le \min \{6p_g-5q+3g_2+2,7\chi+g_2-1\}.$$
\end{theorem}

\begin{corollary}\label{f23}
\cite[p.18]{xiao1} $K_X^2 \le 8\chi(\Oo_X)$.
\end{corollary}

Corollary \ref{f23} shows that the invariants of $X$ are far from extremal: for surfaces of general type the bound is 
$K_X^2 \le 9\chi(\Oo_X)$ \cite[Ch.\thinspace VII, \S 4]{bhpv}.

There is a very long description of surfaces with $q=g_2+1$ in \cite[\S 3]{xiao1} and
almost all genus $2$ fibrations have $q=g_2$. For such fibrations the upper bounds in Thm.\thinspace \ref{f22} give
$$K_X^2 \le \min \{6\chi+4(g_2-1),7\chi+g_2-1\}.$$
Moreover, \cite[p.\thinspace 19]{xiao1} gives a picture showing some forbidden parts on the plane $(K_X^2,\chi)$.
In a small range of integers $g_2$, $\chi:= \chi(\Oo_X)$ and $\epsilon>0$ all numerical invariant are the numerical invariants of some genus $2$ fibration.
\begin{theorem}\cite[Thm.\thinspace 2.9]{xiao1}
Fix integers  $x$, $g_2\geq 0$,  $\epsilon\geq 0$, $\chi\ge g_2-1$ such that 
 $$\epsilon \le \chi-g_2+1  \quad \text{and}\quad \epsilon\equiv \chi +g_2-1\pmod{2}.  $$
  If $\epsilon \neq 0$ assume that $x$ satisfies the inequalities of case  (ii) of Theorem \ref{f22} (with $x$ in the place of $K_X^2$). 
  Then there is a genus $2$ fibration $X\to D$ with $\chi(\Oo_X)=\chi$, $K^2_X=x$, $g(D) =g_2$ and 
  with $\epsilon$ as the degree of stability.
\end{theorem}

\subsubsection{$c_2(X)$ and $c_1(X)^2$}

Let $X$ be a smooth connected complex surface and $\pi\colon \widetilde X\to X$ be the blow up of $X$ at one point
 \cite[p.\thinspace 473]{gh}. Then \cite[p.\thinspace576]{gh} shows that
 $$c_2(\widetilde X) =c_2(X)+1, \quad \quad c_1(\widetilde X)^2= c_1(X)^2 -1, \text{and}$$
  $$c_1(X)^2+c_2(X) =12\chi (\Oo_X) =12\chi(\Oo_{\widetilde X})=c_1(\widetilde X)^2+c_2(\widetilde X).$$ 

 Thus,
  one often assumes that the surface $X$ has no negative curve of the first kind, i.e. no curve $T\subset X$ such that $T\cong \PP^1$ and $T^2=-1$. Indeed, if any such $T$ exists we may blow it down to
  get another smooth compact surface. Every smooth compact surface has a minimal model, i.e. there is a finite sequence
$u\colon X\to X_1$ of blowing downs of curves of the first kind with $X_1$ a minimal model and 
one may then only study $c_2(X_1)$ and $c_1(X_1)^2$. The situation is a bit different if we wish to study  fibrations.

\begin{remark}
In \cite[\S\thinspace 6]{xiao1} there is a complete list of all surfaces of general type with more than one fibration of genus $2$ curves.
Fix an integer $g\ge 1$. A surface $X$ of general type has at most finitely many fibrations whose general fibre is a 
genus $g$ curve \cite[Prop.\thinspace 6.1]{xiao1}. There is a surface of general type $X$ with fibrations by genus $g$ curves for infinitely many genera $g$ \cite[Ex.\thinspace 6.3]{xiao1}.
For each $g\ge 1$ there is an example of a surface $X$ with $\kappa(X)=1$ 
and infinitely many (countably many)  fibrations whose general fibre is a smooth curve of genus $g$
\cite[Ex.\thinspace 6.2]{xiao1}.
\end{remark}

\begin{remark}\label{ccc1}
There are many compact complex surfaces which admit no non-constant holomorphic maps 
$f\colon X\to D$ with $D$ any smooth compact curve. For instance, no surface with 
$\mbox{Pic}(X)\cong \ZZ$ has such an $f$ for any compact $D$. Accordingly, $\PP^2$ has no such map $f$,
neither do even some K3 surfaces, many surfaces of general type, and
also all complete intersection surfaces. Lefschetz pencils show that, for any $X$, 
there exists a blowing up at finitely many points $u\colon  \widetilde X\to X$ such that
 $\widetilde X$ has a surjection $f\colon \widetilde X\to \PP^1$ with each fibre either smooth or irreducible 
 with only one node. However, $u$ is usually a blow up of $X$ at many points.
\end{remark}

For a surface of {\bf general type} $\boxed{c_1(X)^2\le 3c_2(X)}$, see \cite[Ch.\thinspace VII, Th.\thinspace 4.1]{bhpv}, and for {\bf minimal surfaces of general type} $\boxed{c_1(X)^2>0}$.\\

Now we consider surfaces (minimal or not) with a fibration. 

\begin{remark}
Let $f\colon X\to D$ be a minimal fibration as in Def.\thinspace\ref{min}, with $D$ a smooth curve of genus $g_2>0$. Since every holomorphic map $\PP^1\to D$ is constant, $X$ contains no rational curve not even singular ones.
Indeed, suppose there is a, possibly singular, rational curve $T\subset X$ and call $w\colon \PP^1\to T$ the normalisation map.
Since $f\circ w\colon \PP^1\to D$ is constant, $T$ is contained in a fibre of $f$. Thus if $f$ is a minimal, then $X$ is minimal.
 In particular $X$ has no exceptional curve of the first kind, i.e. it is a minimal surface as an abstract surface.
 \end{remark}

For general properties of surfaces, we recommend \cite[Ch.\thinspace IV]{bhpv}, in particular the first 6 sections. If we allow the non-algebraic case and call $a(X)\in \{0,1,2\}$ the {\it transcendental degree} of the field $\CC(X)$ over $\CC$, 
then the existence of a morphism $f\colon X\to D$ implies that $X$ contains a one-parameter family of curves and hence $a(X)>0$. 
The case $a(X)=1$ may arise only if $D$ has genus $0$ and the fibres of $f$ have genus $1$, so a case not of interest in this subsection.
 We have $a(X)=2$ if and only $X$ is projective \cite[Cor.\thinspace IV.6.5]{bhpv}. \\

 We always have $q(X)\le g_1+g_2$, where $g_1$ is the genus of the general fibre of $f$ and $g_2$ is the genus of the base,
and $$\boxed{\chi(\Oo_X) \ge 2(g_1-1)(g_2-1)},$$ 
see \cite[Ch.\thinspace III, Cor.\thinspace11.6]{bhpv}, and for further details see \cite[Ch.\thinspace III, Prop.\thinspace 11.4(ii)]{bhpv} where the exact contribution of the singular fibres is computed.

If both $g_1\ge 2$ and $g_2\ge 2$, then $X$ is of general type, 
because in such cases any rational or elliptic curves are contained in fibres of $f$ and there are only finitely many of those. In these cases we also have $q(X)\ge g_1\ge 2$.
Now assume (any $g_1$) that $X$ is of general type.
 If we call $X'$ the minimal reduction of $X$, we have not only $c_1(X')^2\le 3c_2(X')$ 
 (which is better than just using the same inequality for $X$), but also the Noether inequality $$p_g(X')\le \frac{1}{2}c_1(X')^2 +2$$ \cite[Ch.\thinspace VII, Thm.\thinspace 3.1]{bhpv}, with $p_g(X)=p_g(X')$. Surfaces $X'$ for which equality holds are called surfaces on the Noether's line. Recall also that $q(X') =q(X)$. As  corollaries one obtains other inequalities and a sufficient condition for having $q(X')=0$ \cite[IV. Cor.\thinspace 3.2\&3.3]{bhpv}.
 
 We also recall the Albanese mapping $\alpha_X\colon X \to \mathrm{Alb}(X)$ with $\mathrm{Alb}(X)= (\mathrm{Pic}_0 X)^\vee$ 
 a complex compact torus (it is algebraic because $X$ is algebraic and hence K\"{a}hler and in such 
 a case $\dim \mathrm{Alb}(X) =q(X)$ \cite[pp.\thinspace 44--47]{bhpv}.
 If $\alpha_X(X)$ is a curve, then the curve $\alpha_X(X)$ is smooth, connected and of genus $q(X)$ \cite[Cor.\thinspace I.13.9(iii)]{bhpv}.

Since any morphism from $\PP^1$ to a compact complex torus is constant, $\alpha_X(T)$ is a point for every rational curve $T\subset X$ (even singular $T$, but with $\PP^1$ as its normalisation), if $u\colon X\to X'$ is the minimal reduction of $X$, then $\alpha _X = \alpha_{X'}\circ u$.
 
 \begin{remark}\label{cca1}
 Assuming $q(X)=1$, then we have that $\mathrm{Alb}(X)$ is an elliptic curve and $\alpha_X\colon X\to \mathrm{Alb}(X)$ is a fibration (i.e. it has connected fibres and only finitely many singular fibres). A general fibre may be rational, but in such a case $X$ is birational to $D\times \PP^1$. 
 All possible genera may occur as the genus of a general fibre of $\alpha_X$.
 \end{remark}

We always have $q(X)\ge g_2$, since $f^\ast\colon \mathrm{Pic}(D) \to \mathrm{Pic}(X)$ is injective, 
given that the fibres of $f$ are connected. Thus, a necessary condition to have a fibration with target of positive genus is that $q(X) >0$. By Remark \ref{cca1} all $X$ with $q(X)=1$ have a fibration with $g_2=1$ and the fibration is unique, up to isomorphisms of $X$ and the base.

For fibrations by curves of genus 3 or higher see \cite{bs}.

 \subsection{Examples of hyperelliptic fibrations with nodes}\label{example-eisen}
By definition, a smooth curve $C$ of genus $g\ge 2$ is  hyperelliptic if there is a degree $2$ morphism $u: C\to \PP^1$. 
Equivalently, by the universal property of the projective line, 
$C$ is hyperelliptic if and only if there exist a degree $2$ line bundle $L$ on $C$ and a $2$-dimensional linear space $V\subseteq H^0(L)$ with no base points (hence inducing the map $u$). 
  
\begin{figure}[h]
\includegraphics[scale=0.18]{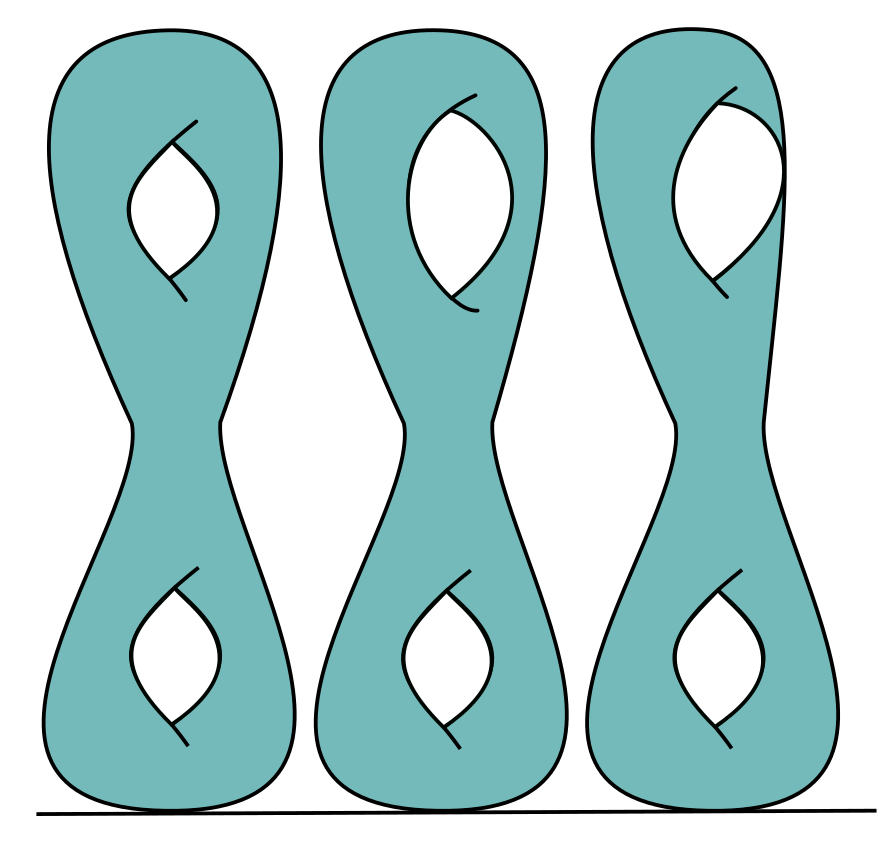}
\caption{Hyperelliptic family with irreducible singular fibre}
\label{fig01}
\end{figure}  
  
  Around 1817, N. Abels showed that such genus $g$  curves may be understood as follows. 
  Fix a degree $2g+2$ polynomial $f(x)\in \CC[x]$ without multiple roots. Consider the affine curve 
  $$y^2=f(x).$$ Its unique smooth projective completion is a hyperelliptic curve of genus $g$, the morphism $u$ being induced by the projection $(x,y)\to x$. If we view $\CC^2$ as an affine chart of the plane $\PP^2$, then the corresponding closure of the affine curve is singular.

 A natural ambient for  hyperelliptic curves is the weighted projective plane $\mathbb P(1,1,g+1)$. We take coordinates $x_0,x_1,y$ 
 and call $h(x_0,x_1)$ the homogeneous degree $g+1$ polynomial such that $h(1,x_1) =f(x)$. 
 Giving weight $g+1$ to the variable $y$, the equation $y^2=h(x_0,x_1)$ defines our curve in $\mathbb P(1,1,g+1)$.
  Furthermore, fixing any integer $t$ such that $1\le t\le g$, we may take as $f(x)$ a degree $2g+2$ 
  polynomial with $2g+2-2t$ simple roots and $t$ roots of multiplicity $2$. Then the curve $y^2=h(x_0,x_1)$ is an irreducible and nodal curve of arithmetic genus $g$ with exactly $t$ nodes.\\
 
 Now we explain two other constructions 
 containing all smooth genus $g$ hyperelliptic curves
 within an ambient surface (here a smooth surface)  and, containing
  in the same linear system (so the same arithmetic genus $g$) irreducible curves with exactly $t$ nodes for all $t=1,\dots ,g$.\\
 
 \quad  {\bf First construction:} Set $F_0:= \PP^1\times \PP^1$. We have $\mathrm{Pic}(S)\cong \ZZ^2$ and 
 we may take as a generators of $\mathrm{Pic}(F_0)$ the isomorphism classes $\Oo_{F_0}(1,0)$ and $\Oo_{F_0}(0,1)$
 of the fibres of the $2$ projections $F_0\to \PP^1$. For all $(a,b)\in \NN^2$ we have $ h^0(\Oo_{F_0}(a,b)) =(a+1)(b+1)$.
 
 Now assume $a>0$ and $b>0$. In this case $\Oo_{F_0}(a,b)$
 is very ample. Hence, a general element $|\Oo_{F_0}(a,b)|$ is smooth and connected.
  We have $h^0(\Oo_C)=1$ 
 for any $C\in |\Oo_{F_0}(a,b)|$, even for the ones with multiple components. 
 
 Since $\omega_{F_0}\cong \Oo_{F_0}(-2,-2)$, $\omega_C\cong \Oo_C(a-2,b-2)$
 for any $C\in |\Oo_{F_0}(a,b)|$. Thus, all $C\in |\Oo_{F_0}(a,b)|$ have arithmetic genus $1+ab-a-b$. \\
 
 \quad (a) {\bf The parameter space:}
 Fix an integer $t$ such that $0\le t \le 1+ab-a-b$. Let $V(a,b,t)$ denote the set
 of all irreducible and nodal $C\in |\Oo_{F_0}(a,b)|$ with exactly $t$ nodes. 
 We have $$V(a,b,t)\ne \emptyset,$$ and  $V(a,b,t)$ is irreducible with $\dim V(a,b,t) =\dim |\Oo_{F_0}(a,b)|-t
 =ab+a+b-t$, see \cite{tan,ty}. 
 Take $a=2$ and $b=g+1$. Each $C\in |\Oo_{F_0}(2,g+1)|$ has arithmetic genus $g$. 
 Each smooth $C$ is hyperelliptic (use either of the $2$ projections $F_0\to \PP^1$). 
 
 Let $D$ be a smooth hyperelliptic curve of genus $g$. Fix a general $L\in \mathrm{Pic}^{g+1}(D)$. Since $g\ge 2$ and $L$ is general, $h^0(L)=2$ and $L$ is base point free.
 Thus $|L|$ induces a degree $g+1$ morphism $v: D\to \PP^1$. Let $u: D\to \PP^1$ be the degree $2$ map given by the hyperellipticity of $D$ and $R\in \mathrm{Pic}^2(D)$ the associated line bundle. Let $w =(u,v)\to \PP^1\times \PP^1$ be the morphism induced by $u$ and $v$. Since $L$ is not a multiple of $R$ and $4\deg (u)=2$, $w$ is birational onto its image. The curve $\mathrm{Im}(w)\in |\Oo_{F_0}(2,g+1)|$ has arithmetic genus $g$ and the smooth genus $D$ curve as its normalization. Thus $\mathrm{Im}(w)\cong D$
 and $w(D)\in |\Oo_{F_0}(2,g+1)|$.
 
 In general, we may take any Hirzebruch surface $F_e$, $e\ge 0$ \cite[Ch.\thinspace V,\S 2]{h}, and as in step (b) below we may use explicit equations for the linear systems in $F_e$,  many of which contain  hyperelliptic curves.  
 For any $0\le t\le p_a$, where $p_a$ is the arithmetic genus of any element of $|\Oo_{F_e}(ah+bf)|$, the
 existence, irreducibility, and the dimension  of the family all irreducible nodal curves in a prescribed linear system $|\Oo_{F_e}(ah+bf)|$ with exactly $t$ nodes 
 are given in \cite{tan,ty}.\\

 \quad (b) {\bf Hyperelliptic families with irreducible fibres having 1 or 2 nodes}: Fix an integer $g\ge 2$. Let $\Hh(g)$ be the stack of all smooth hyperelliptic curves of genus $g$. 
We recall that any smooth genus $2$ curve is hyperelliptic.
For any $X\in \Hh(g)$ call $h_X\colon X\to \PP^1$ the degree 2 morphism, usually denoted $g^1_2$,
corresponding (by definition) to the hyperelliptic curve  $X$.

We recall that such degree 2 morphism is unique for each hyperelliptic curve of genus $g>1$
and for $g=2$ it is the canonical map. Denote by $R$ the only spanned degree $2$ line bundle on $X$, so that
$h_X$ is the morphism associated to the complete linear system $\vert R\vert$.
Since $\deg (L) =2g+2\ge 2g+1$, the line bundle $L:= R^{\otimes (g+1)}$ is very ample and non-special. Thus $h^0(L)=\deg(L)+1-g$ and
$|L|$ induces an embedding $f\colon X \to \PP^{g+1}$. Fix $D_1,D_2,D_3\in |R|$ such that $D_i\ne D_j$ for all $i\ne j$. Since $f$ is an embedding and $\deg (D_i)=2$, $f(D_i)$ spans a line $\langle D_i\rangle$. The line bundle $R^{\otimes (g-1)}$ induces a degree $2$ map with as image $\PP^1$ (case $g=2$) or a rational normal curve of $\PP^{g-1}$ (case $g\ge 2$). We have $h^0(R^{\otimes g-2}) = g-1$. We see that the lines
$\langle D_i\rangle$ and $\langle D_j\rangle$ span a plane and hence they meet. 
The $3$ lines $\langle D_1\rangle \cup \langle D_2\rangle \cup \langle D_3\rangle $ span a $\PP^3$, because $h^0(R^{\otimes g-2}) = g-1 =h^0(R^{\otimes (g+1)}) -4$. 
We obtain that $f(X)$ is contained in a cone $T$ with vertex $o\notin f(X)$ and as a base a rational normal curve of $\PP^{g+1}$,
 see \cite{eis}. Let $u\colon Y\to T$ be the minimal resolution of $T$. The surface $Y$ is isomorphic to the Hirzebruch surface $F_{g+1}$ and $h:= u^{-1}(o) \cong \PP^1$. We have $\mathrm{Pic}(Y)\cong \ZZ^2$ with $h$ and a fibre $f$ of its ruling (i.e. the strict transform of a line of $T$ passing through $o$) as a basis over $\ZZ$,
with intersection numbers $$f^2=0, \quad h\cdot f =1, \quad \text{and} \quad h^2=-g-1.$$ 
 For any curve $D\subset F$ (even reducible) which is not a member of $|cf|$ for any $c>0$, there are integers
$a>0$ and $b\ge a(g+1)$ such that $D\in |\Oo_Y(ah+bf)|$. 
The linear system $|\Oo_Y(ah+bf)|$ contains a curve $D$ such that $D\cap h=\emptyset$, i.e. such that $o\notin u(D)$ if and only if $b =a(g+1)$. 

The case $a=2$ and $b=2g+2$ is the linear system corresponding to our genus $g$ hyperelliptic curves. 
We take an arbitrary integer $a>0$ and give an explicit parameter space for the part of the linear system $|\Oo_Y(ah+a(g+1)f)|$ not intersecting $h$.

 Fix variables $x_0,x_1,y$ and give weight $1$ to the variables $x_0$ and $x_1$ and weight $g+1$ to the variable $y$. Let $V(g+1,a)$ denote the set of all $f\in \CC[x_0,x_1,y]$ which are weighted homogeneous with total degree $a(g+1)$. The zero-locus of any $v\in V(a+1,g)$ is an element $D\in |\Oo_Y(ah+(g+1)f|$ such that
$D\cap h=\emptyset$.
Those belonging to the linear system $|\Oo_Y(ah+a(g+1)f)|$ such that $D\cap h\ne \emptyset$ have $h$ as a component,
 because the intersection number
of $h+(g+1)f$ and $h$ is $0$. 

Take again $a=2$. In this case we get a linear system whose smooth elements are smooth genus $g$ hyperelliptic curves.
 Their nodal degenerations may contain $t$ nodes, for any $0\le t\le g$ by \cite{tan,ty} and also the reducible nodal curves of the form $D_1\cup D_2$ with $D_1\cong D_2\cong \PP^1$ and $\#(D_1\cap D_2)=g+1$.  It is sufficient to take as $D_1$ and $D_2$ two general elements of $|h+(g+1)f|$.\\
 
 \begin{figure}[h]
\includegraphics[scale=0.18]{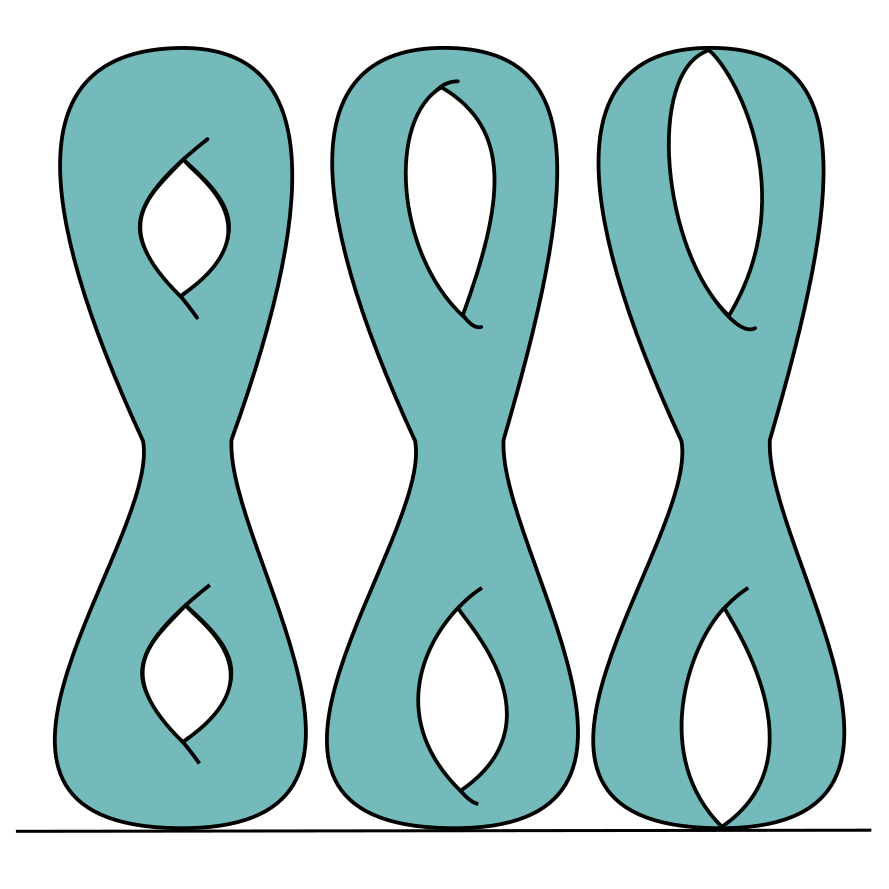}
\caption{Family with 1 irreducible fibre having 2 nodes}
\label{2nodes}
\end{figure}

 \quad (c) {\bf Hyperelliptic families with reducible nodal fibres}: For an integer $g\ge 2$,
  let $\overline{\Mm}_g$ denote the coarse moduli space of stable genus $g$ curves. In the boundary $\overline{\Mm}_g\setminus \Mm_g$ we get all nodal reducible curves $C\cup E$ with $C$ a smooth curve of genus $g$, $E$ an elliptic curve and $C\cap E$ a unique point.
 If $g=2$ we get a nodal union of $2$ elliptic curves which 
 occurs as a limit of a family of smooth genus $2$ curves. 
 Now assume $g\ge 3$ and assume that $C$ is hyperelliptic. Both the theory of generalised coverings \cite{hm} 
 and that  of limit linear series \cite{EH} give that $C\cup E$ is the flat limit of a family of hyperelliptic curves. See also
  \cite[Ch.\thinspace 6 C]{hi}.
 In $\overline{\Mm}_g$ there  exists a point representing the nodal reducible curve $D_1\cup D_2$ described at the end of part (b).

\subsection{Examples of genus 2 fibrations with nodes}\label{example-node}
We use another construction to build more examples of hyperelliptic fibrations with nodes.
\subsubsection{Case $S\ne \emptyset$} 

 {\bf Second construction}: 
  Let $X$ be a smooth Del Pezzo surface of degree $1$ \cite[Ch.\thinspace 8]{do}.
This smooth surface $X$ may be realised as a blowing up of $\PP^2$ at a set $S\subset \PP^2$
of $8$ points, with the conditions that no $3$ of the points of $S$ are colinear, no $6$ of the points of $S$ are contained in a conic and
no cubic surface passing through each point in S that is singular at one of them.

It is  most useful  to see $X$ as a sextic hypersurface in the weighted projective space 
$\PP(1,1,2,3)$ (We recall here that this notation of the $4$ coordinates weights means that 
$\lambda[x_0,x_1,x_2,x_3] = [\lambda x_0,\lambda x_1,\lambda^2  x_2,\lambda^3x_3]$.

 For any integer
  $r\ge 1$ we have $$\dim |-rK_X| =r(r+1)/2,$$ 
  this is the case $d=1$ of \cite[Lemma 8.3.1]{do}. Hence $\dim |-K_X|= 1$ and $\dim |-2K_X|=3$.
The pencil $|-K_X|$ had a unique base point $p$, and a general element of $|-K_X|$ is a smooth elliptic curve. Every element of $|-2K_X|$ has arithmetic genus $3$, a general element of $|-2K_X|$ is smooth and among the elements of $|-2K_X|$ there are two curves $E_1\cup E_2$ with $(E_1,E_2)$ general in $|-K_X|$.\\

\quad (a) {\bf Existence of reducible curves with one node}: 
 So we have a $2$-dimensional family of nodal unions of $2$ elliptic curves with a unique node at $p$, 
inside the $3$-dimensional projective space $|-2K_X|$. 
 This is illustrated in Figure \ref{red1node}.
\begin{figure}[h]
\includegraphics[scale=0.18]{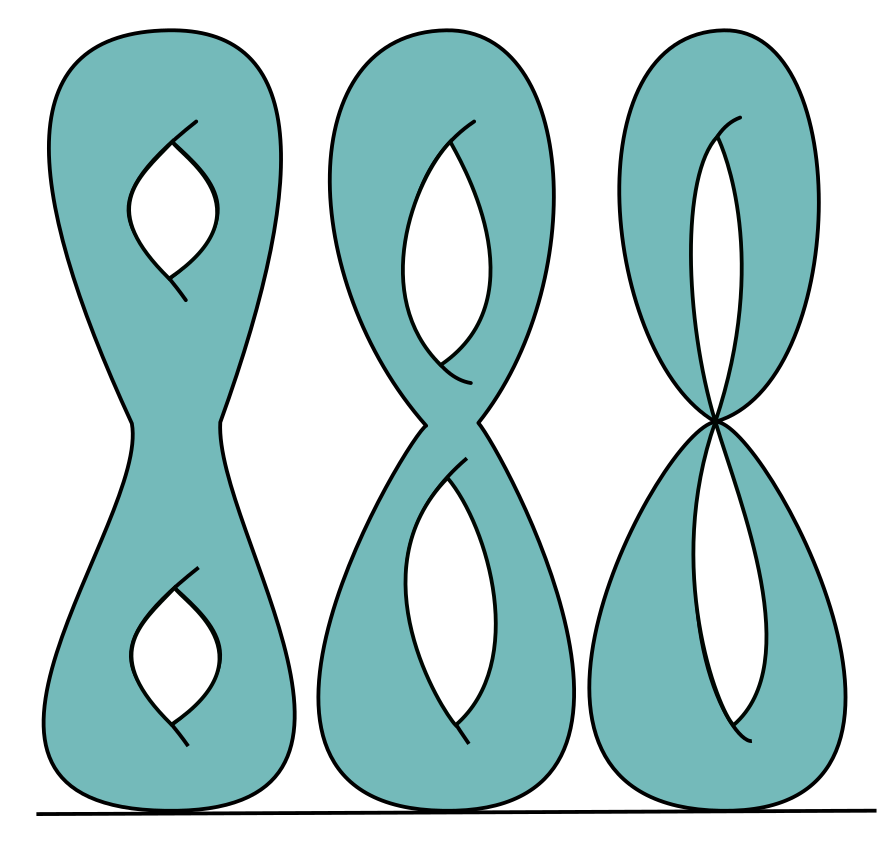}
\caption{Family with node on a reducible fibre}
\label{red1node}
\end{figure}

A general line in $|-2K_X|$ gives a fibration with $\PP^1$ as its base.
This type of construction, which   was earlier considered by  Halphen, is described  further details on the section of Halphen's pencils 
 in \cite[Ex. 7.20]{do}.\\

\quad (b) {\bf Existence of irreducible curves with nodes on $S$}: 
This paragraph is due to the referee.
The linear system $|-2K_X|$ gives a double cover $\pi\colon X \rightarrow \mathbb P(1, 1, 2)$ 
branched over the vertex of the quadric cone  $\mathbb P(1, 1, 2) \subset  \mathbb P^3$ 
and a smooth curve $R$ of degree $6$.
 Take general point $P$ in X that is mapped to R, and let $\mathcal P$ 
  be the pencil in$|-2K_X|$  that consists of pre-images of hyperplane sections of the cone $\mathbb P(1, 1, 2)$
   that are tangent to $R$ at $\pi(P)$.
    Then a general member of the pencil $\mathcal P$ 
     is irreducible and has a node at $P$. 

\begin{figure}[h]
\includegraphics[scale=0.18]{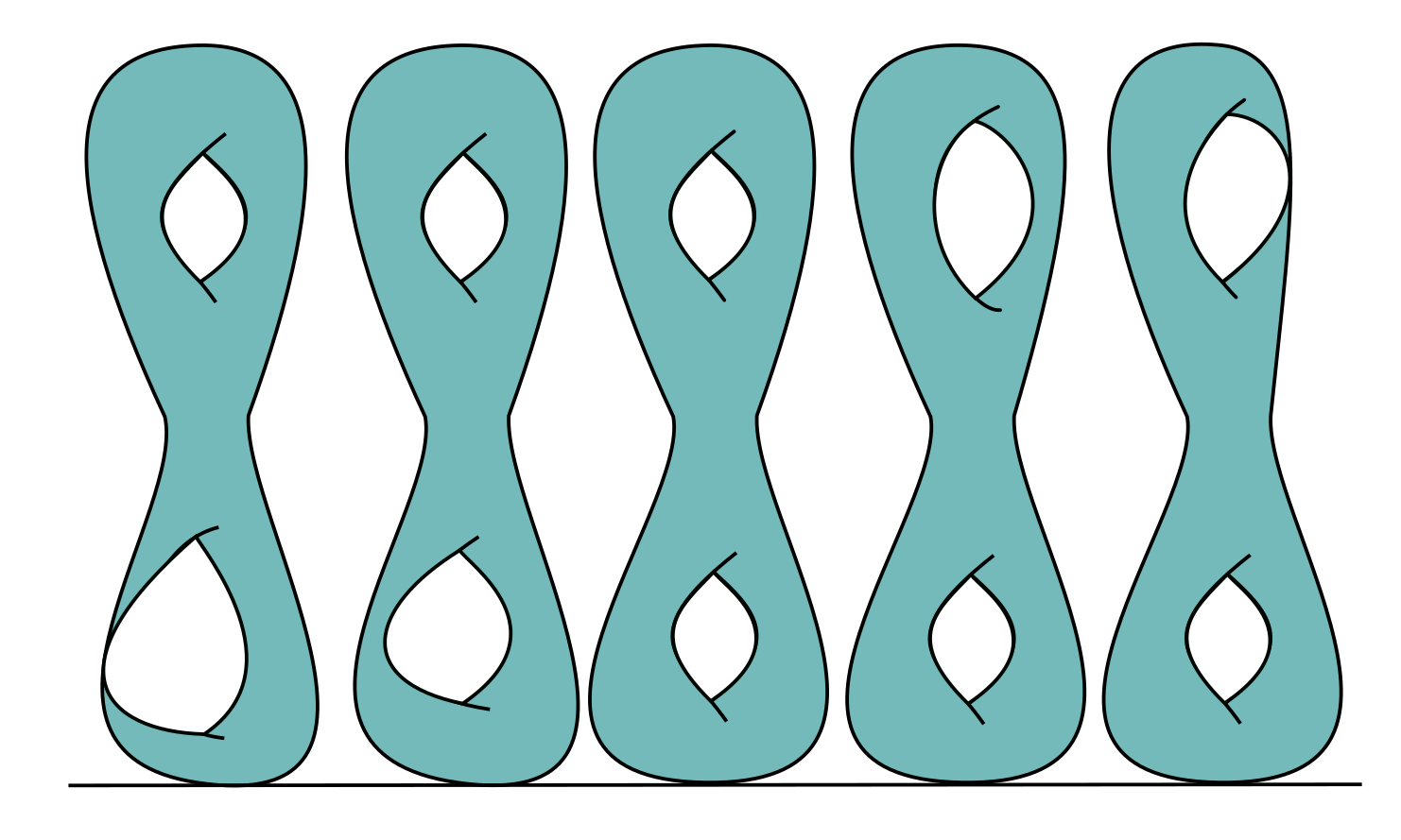}
\caption{Family with 2 nodes on different fibres}
\label{2nodes2fibres}
\end{figure}

\subsection{Families of hyperelliptic curves in $\PP^1\times \PP^1$ }\label{toric-families}

For completeness we present some further details of the toric case.
However, we observe that our constructions provide genus 2 unions of elliptic curves inside toric varieties
only in the case just described above, which is a very particular construction on a blowing up of $\mathbb P^2$.\\

Fix integers $g\ge 2$ and $0\le c\le g$ and a general $S\subset \PP^1\times \PP^1$ with $\#S=c$. Now
 we consider the zero-dimensional scheme $Z$ defined by 
 $$Z:= \cup _{p\in S}2p.$$
 
 Elements of $|\Oo_{\PP^1\times \PP^1}(2,g+1)|$ containing $Z$ are exactly the elements
of $|\Oo_{\PP^1\times \PP^1}(2,g+1)|$ singular at all points of $S$. We have $h^0(\Oo_{\PP^1\times \PP^1}(2,g+1)) = 3(g+2)$.
It is well-known that $h^0(\Ii_Z(2,g+1)) =3(g+2)-3c$ (\cite{l}), i.e. $$\dim | \Ii_Z(2,c)| =3g+5-3c$$
 and that a general $D\in |\Ii_Z(2,c)|$ is an integral curve of arithmetic genus $g$, geometric genus $g-c$ and exactly $c$ ordinary nodes as singularities. If $S$ is defined over $\RR$
we may even find $D$ defined over $\RR$. 

A general $C\in |\Oo_{\PP^1\times \PP^1}(2,g+1)|$ is smooth. Since $\omega _{\PP^1\times \PP^1}\cong \Oo_{\PP^1\times \PP^1}(-2,-2)$, the adjunction formula gives $\omega _C\cong \Oo_C(0,g-1)$.
So, $\deg (\omega _C)=2g-2$ and  $C$ has genus $g$.

Therefore, we get a $3g+5$-dimensional family of smooth genus $g$ curves. All of them are hyperelliptic.
Indeed, the $2$-to-$1$ morphism $C\to \PP^1$
is induced by the projection $\PP^1\times \PP^1$ onto one of its factor induced by the linear system $|\Oo_{\PP^1\times \PP^1}(0,1)|$ (we have $h^0(\Oo_{\PP^1\times \PP^1}(0,1))=2$ and the intersection number $(0,1)\cdot C$ is $2$, because
$(0,1)\cdot (2,g+1) =2$.

\begin{remark}
The linear system $|\Oo_{\PP^1\times \PP^1}(2,g+1)|$, $g\ge 2$, contains all smooth genus $g$ hyperelliptic curves.
The linear system $|\Oo_{\PP^1\times \PP^1}(1,1)|$ embeds $\PP^1\times \PP^1$ as a smooth quadric surface $Q$. In these examples each smooth $C\in |\Oo_{\PP^1\times \PP^1}(2,g+1)|$ is embedded as a degree $g+3$ non-special curve in $\PP^3$. Conversely, take any smooth $C\subset Q$ of degree $g+3$ and genus $g\ge 2$, say $C\in |\Oo_{\PP^1\times \PP^1}(a,b)|$. Because $g>0$, we have $a\ge 2$ and $b\ge 2$. 
The isomorphism  $\Oo_{\PP^3}(1,1)_{|Q} \cong \Oo_{\PP^1\times \PP^1}(1,1)$ gives $g+3=\deg(C) =a+b$. 
Since $\omega _{\PP^1\times \PP^1}= \mathcal O_{\PP^1\times \PP^1}(-2,-2)$, the adjunction formula
gives $\omega_C\cong \Oo_C(a-2,b-2)$ and consequently $g =ab-a-b+1$. Up to a change of the two factors of $\PP^1\times \PP^1$ we may assume $a\le b$. We get $a=2$ and $b=g+1$.
Thus, curves of degree $g+3$ and genus $g$ contained in $Q$ are hyperelliptic.
\end{remark}

 \section{\bf  Fibrations with smooth fibres}\label{smooth-fibres}

\vspace{3mm}
 
We describe the case of surfaces.
Let $f\colon X\to D$ be a holomorphic submersion with $X$ a smooth surface and $D$ a smooth (even open) Riemann surface. 
They are described in \cite[Ch V \S 4--6]{bhpv}. Let $g_1$ be the genus of the fibre, call $Y$ the isomorphism class of a fibre.

The easy case of ruled surfaces is described in \cite[p.\thinspace 189-192]{bhpv}, followed by the case of 
elliptic fibre bundles 
\cite[p.\thinspace 193--198]{bhpv}, with classification when $D \cong \PP^1$ and many results (even classifications of some subcases) for $D$ of genus $1$.

As for the case of higher genus fibre bundles, 
when $Y$ has genus $g\ge 2$, $\mathrm{Aut}(Y)$ is finite (indeed $\#\mathrm{Aut}(Y)\le 84(g-1)$ by Hurwitz automorphism theorem).
Every fibre bundle over the smooth curve $D$ with $Y$ as a fibre is given by a representation $\pi _1(D)\to \mathrm{Aut}(Y)$. There is a finite unramified covering $u\colon \tilde{D}\to D$ such that making the fibre product with $u$
we get a fibre bundle $\tilde{f}\colon \tilde{X}\to \tilde{D}$ with 
$\tilde{X} =Y\times \tilde{D}$ and $\tilde{f}$ the projection onto the second factor \cite[pp.\thinspace 199--200]{bhpv}.

\subsection{Smooth semistable fibrations}
 In this section we consider certain fibrations whose general fibre is a smooth curve of genus $g\ge 2$ following 
 \cite[\S III.10]{bhpv}. We allow singular fibres, but only with ordinary double points.
\begin{definition}\label{a001}
A {\bf semistable fibration} $f\colon X\to D$ is a fibration with connected fibres, (i.e. $f_\ast(\Oo _X)=\Oo_D$) with $X$ a smooth and connected complex surface (even not compact) $D$ a Riemann surface (not necessarily compact or algebraic), a general fibre of $f$ is a smooth curve of genus $g\ge 2$, 
 all fibres have at most ordinary double points as singularity and no fibre contains a $(-1)$-curve. 

Take $f\colon X\to D$ semistable as in Definition \ref{a001}; $X$ may have $(-1)$-curves, say  $J$ (not contained in a fibre by assumption), but in this case $D$ is compact, $D$ has genus $0$ and $f_{|J}: J\to D$ is a finite map. In all other cases $X$ is a minimal surface.
\end{definition}

The reason that Definition \ref{a001} requires that $X$ has no $(-1)$-curve contained in a fibre is that if it has one you contract it and get $f_1\colon X_1\to D$ with all the other properties and $X\to X_1$ a blowing up of one point.

\begin{remark} A statement analogous to Theorem \ref{a3} is no longer true when the base is allowed to have genus higher than 1. 
Indeed, the famous Kodaira surfaces, which we discuss in the following section, are counterexamples to such a statement. 
\end{remark}

\subsection{Fibrations with general fibre hyperelliptic}

Let $f\colon X\to D$ be a fibration whose fibres are smooth hyperelliptic curves. In this case there is a monograph \cite{xiao1} proving certain ranges of Chern numbers are allowed/not allowed for $X$.

There are similar results for fibrations $f\colon X\to D$ whose general fibre is a smooth hyperelliptic curve; the most important modern result was proved by Gujar, Paul and Purnaprajna in \cite{gpp}, showing that $f$ has at most multiple fibres of multiplicity $2$ and that if $g$ is even no multiple fibre at all 
\cite[Thm.\thinspace 2]{gpp}.

\begin{remark}\label{hyperd}
If $f\colon X\to D$ is a hyperelliptic fibration, then there is a rational map $\Phi\colon X\dasharrow P$, 
where $u\colon P \to D$ is a $\PP^1$-bundle on $D$ and $v\circ \Phi =f$, hence the diagram 

$$\xymatrix{
X \ar[d]_f \ar@{-->}[r]^{\phi} & P \ar[dl]^u \\
D 
}$$
commutes, but $\Phi$ is only a rational map generically $2$ to $1$, which on a smooth fibre $F$ of $f$ is the degree $2$ morphism coming from the definition of hyperelliptic curve.
\end{remark}

 An observation for the case of smooth compact complex surfaces; if the genus of the fibre satisfies 
 $g_1\ge 2$ we may drop the projectivity of $X$, because it is a consequence of the existence of the fibration in curves of genus at least $2$ 
\cite[Ch.VI\thinspace \S5]{bhpv}.
The relative canonical sheaf of $f$ (usually denoted with $\omega _{X/D}$) is the line bundle 
$\omega _{X/D}\ce \omega _X\otimes f^{\ast}(\omega _D)$. Often you see $K_X$ instead of $\omega _X$ and $K_D$ instead of $\omega _D$; they are the same line bundles, just a different notation.
The sheaf $E\ce f_{\ast}(\omega _{X/D})$ is a rank $g$ vector bundle on $D$. It has the following type of nonnegativity.

\begin{theorem}
\cite[Thm.\thinspace 1.1]{xiao1} Every vector bundle $F$ on $D$ such that there is a surjection $E\to F$ has nonnegative degree.
\end{theorem}
Let $g_2$ the genus of $D$. 
A. Beauville 
proved the second inequality in \cite[p.\thinspace 345]{bea}, while the first inequality in the following lemma is obvious.

\begin{lemma}
We have $g_2 \le q \le g_1+g_2$. Moreover, $q=g_1+g_2$ if and only if $X$ is birational to $D\times C$ with $C$ a curve of genus $g_1$.
\end{lemma}

Fibrations by smooth hyperelliptic fibres are not very varied, more  precisely by \cite[Proposition 2.10]{xiao1}:

\begin{figure}[h]
\includegraphics[scale=0.18]{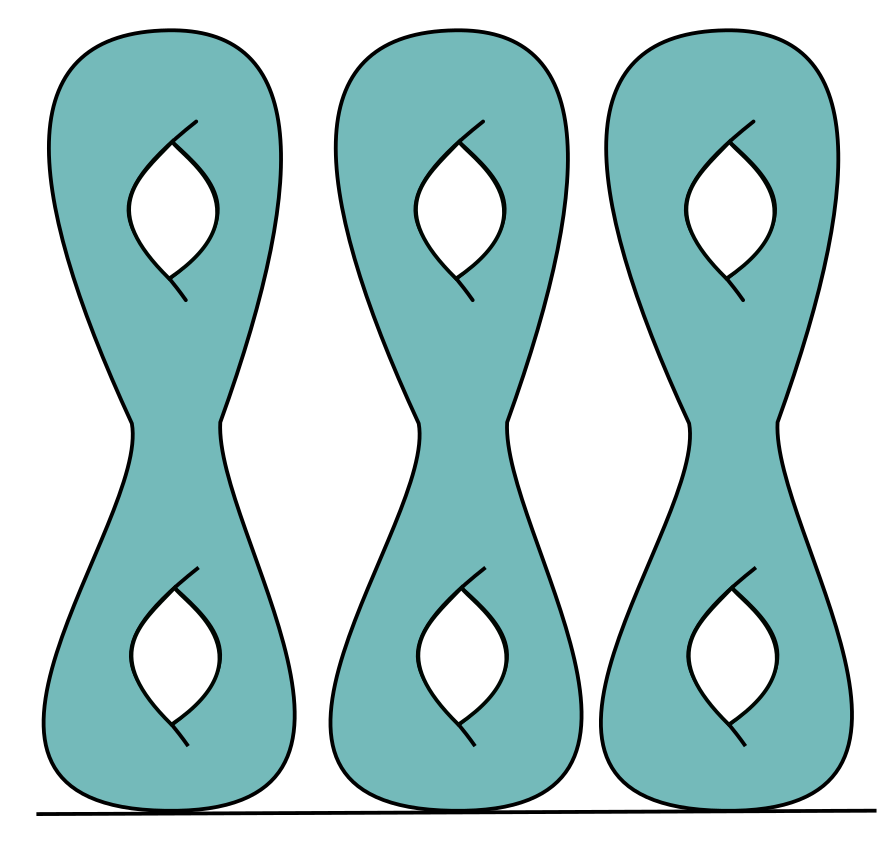}
\caption{Smooth hyperelliptic fibration}
\label{smooth}
\end{figure}

\begin{proposition}
Let $D$ be a smooth projective curve, $S$ a smooth compact surface and $f\colon S\to D$ be a morphism such that all its fibres are smooth hyperelliptic curves of genus at least $ 2$. Then $f$ is isotrivial.
\end{proposition}

Recall that isotrivial means that the smooth fibres are all isomorphic, as in Figure \ref{smooth}. For more details on isotrivially fibred surfaces see \cite{sa,se}.

\subsection{Kodaira fibrations }\label{kod-surf}
Following \cite[Ch.\thinspace V\S 14]{bhpv} we discuss  Kodaira fibrations as examples of families 
of curves 
that behave very differently from  families of hyperelliptic curves. 
\begin{definition}
A  {\bf Kodaira fibration}
 is a smooth compact complex surface $X$ such that there is a submersion $f\colon X\to D$ with $D$ a smooth compact curve, 
{\em all fibres of $f$ are smooth of genus $g$, but  they do not form a locally trivial fibre bundle in the holomorphic category} (of course it is a differentially locally trivial fibre bundle). 
\end{definition}

\begin{figure}[h]
\includegraphics[scale=0.18]{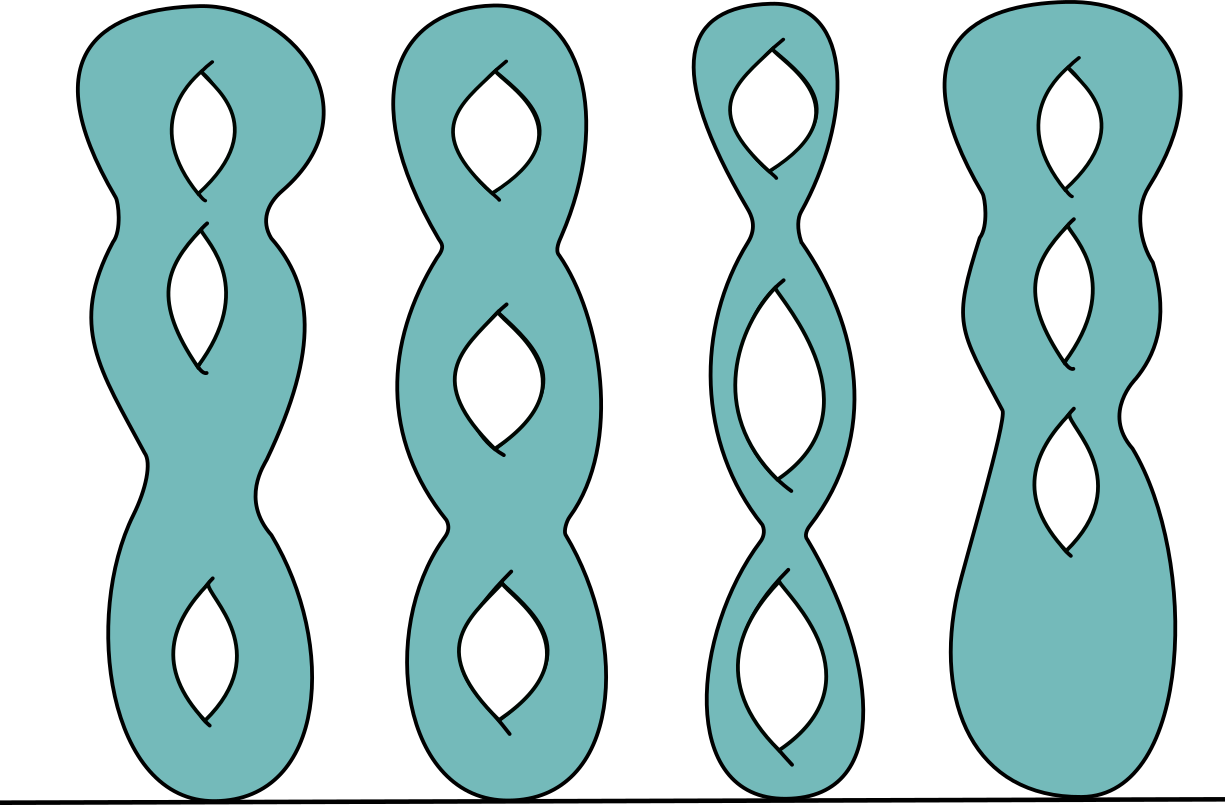}
\caption{Kodaira fibration}
\label{kodfib}
\end{figure} 

\begin{remark}
In view of the Grauert--Fischer Theorem \ref{a2} this means that, though all fibres are smooth curves, their complex structure varies.
It follows immediately from the uniqueness of $\mathbb P^1$ as a curve of genus $0$ and the existence of the 
$J$-fibration \cite[Sec.\thinspace 9]{bhpv} that the fibre genus of a Kodaira fibration is at least $2$. 
Such an inequality  also holds for the base genus. Hence, for a Kodaira surfaces we have $g_1\geq 2$ and $g_2\geq 2$, see
\cite[p.\thinspace 220]{bhpv}. 
\end{remark}

\begin{remark}\label{nuovo}
There does  not exist any Kodaira surface  whose fibres are smooth curves of genus $2$  \cite[Prop.\thinspace 2.10]{xiao1}. 
In further  generality, if  $f\colon X\to D$ is any submersion whose fibres are smooth hyperelliptic curves, then  all fibres 
of $f$ are isomorphic
\cite[Prop.\thinspace 2.10]{xiao1}. Since by Theorem \ref{a2} not being  a fibre bundle is equivalent to not all  fibres being isomorphic,
it follows that fibres of a Kodaira fibration are not hyperelliptic curves.
\end{remark}

 Nowadays, the preferred approach to Kodaira surfaces  uses moduli spaces and  moduli stacks for both curves and  surfaces. 
 The reader may find this approach at its best in \cite{cat} by F. Catanese, in  references therein, and 
 certainly also in  papers quoting \cite{cat}. We summarise a few very interesting facts.
 Let $\Mm_g$, $g\ge 2$, be the coarse moduli scheme of genus $g$ curves. Any genus $g$ Kodaira fibration 
 $f\colon X\to D$ induces a non-constant morphism $w\colon D\to \Mm_g$. Thus, the known fact that there is no genus $2$ Kodaira fibration follows from the fact that $\Mm_2$ is affine, while the existence of genus $g\geq 3$ Kodaira fibrations follows from the 
 well known result that for genus  $g\ge 3$ the  coarse moduli scheme 
 $\Mm_g$ contains projective curves. Furthermore, if  $g\ge 3$ then for a general $[C]\in \Mm_g$ there 
 exists a projective curve $T\subset \Mm_g$ such that $[C]\in T$, see \cite[Thm.\thinspace 2.33]{hi}. 
 Consequently,
  for every $g\ge 3$ there exists a Kodaira fibration containing a general genus $g$ curve.

\begin{lemma}\label{imp}Fix an integer $g\ge 3$. In any genus $g$ Kodaira fibration  each fibre
is isomorphic to only finitely many fibres.
\end{lemma}

\begin{proof} 
Let $\mathcal M_g$ be the coarse  moduli scheme of genus $g$ curves, which is a quasi-projective variety.
 The morphism $f\colon X\to D$ is a flat family of genus $g$ curves. By the universal property of  coarse moduli spaces
the family $f$ induces a morphism $w\colon D\to \mathcal M_g$. If the fibres $F_1$ and $F_2$ are isomorphic as abstract curves, then $w(F_1)=w(F_2)$. Since not all fibres of $f$ are isomorphic, $\dim w(D)=1$. Since  $w$ is an algebraic map, 
if a fibre has dimension 0, its reduction is a finite set. Hence each fibre of $w$ is finite. 
Thus each fibre of $f$ is isomorphic to only finitely many fibres of $f$.
\end{proof}

All Kodaira fibrations are projective surfaces \cite[p.\thinspace 220]{bhpv}, and for any Kodaira fibration there are bounds
$$\boxed{2 <\frac{c_1^2(X)}{c_2(X)} < 3}.$$

The following paragraph is  shamelessly copied  from   \cite{cp}: The number $\nu = \frac{c_1^2(X)}{c_2(X)}$ is an  important invariant of Kodaira fibred surfaces, called  the slope, that can be seen as a quantitative measure of the non-multiplicativity
of the signature. In fact, every product Kodaira surface  satisfies $\nu = 2$; on the other hand, if $S$ is a Kodaira fibred surface, then Arakelov inequality (see \cite{bea}) implies $\nu (S) > 2$, while Liu inequality (see \cite{liu}) yields $\nu(S) < 3$, so that for such a surface the slope lies in the open interval $(2, 3)$. The original examples by Atiyah, Hirzebruch and Kodaira have slope lying in $(2, 2 + 1/3]$, 
(see \cite[p.\thinspace221]{bhpv}), and the first examples with higher slope were given by  Catanese and  Rollenske in \cite{CR} using double Kodaira that  satisfy 
$\nu(S) = 2 + 2/3$. More examples of double Kodaira fibrations where given by Causin and Polizzi in \cite{cp} and further explored by Polizzi in \cite{p}.   At present it is  unknown whether the slope of a Kodaira fibred surface can be arbitrarily close to $3$.

 J. Jost and S.-T. Yau proved that every deformation of a Kodaira fibration is a Kodaira fibration, see  \cite[p.\thinspace 223]{bhpv} and \cite{jy}.\\


\section{\bf Acknowledgements} 
 We greatly appreciate  thorough and comprehensive referee report
which contributed  to improve the quality of our text.
We are grateful to Francesco Polizzi for pointing out 
necessary corrections to text. 
We thank Maria Pilar Garcia del Moral Zabala and Camilo las Heras for asking us to write
this note about existence and numerical invariants of hyperelliptic fibrations. 
Ballico is a member of MUR and GNSAGA of INdAM (Italy).
Gasparim and Suzuki thank the University of Trento for the support and 
excellent hospitality during their visit under the research in pairs program of CIRM.
Suzuki was supported by Grant 2021/11750-7 S\~ao Paulo Research Foundation - FAPESP.
Gasparim is a senior associate the Abdus Salam International Centre for Theoretical Physics, Italy.\\

\vspace{12pt}

\noindent Edoardo Ballico, \texttt{edoardo.ballico@unitn.it}\\
Department of Mathematics, University of Trento\\
Trento, Italy

\vspace{12pt}

\noindent Elizabeth Gasparim, \texttt{etgasparim@gmail.com}\\
Departamento de Matem\'aticas, Universidade Cat\'olica del Norte\\
Antofagasta, Chile

\vspace{12pt}

\noindent Bruno Suzuki, \texttt{obrunosuzuki@gmail.com} \\
Instituto de Matem\'atica e Estat\'istica, Universidade de S\~ao Paulo\\
S\~ao Paulo, Brazil

\end{document}